\documentclass[leqno,12pt,a4paper,twoside]{article}
\usepackage{amsmath,amsthm,indentfirst}
\usepackage{amssymb}
\usepackage{amsfonts}
\usepackage{a4}
\usepackage{amscd}
\usepackage[latin1]{inputenc}

\theoremstyle{plain}
\newtheorem*{maintheorem*}{Main Theorem}
\newtheorem*{thm*}{Theorem}
\newtheorem*{thma*}{Theorem A}
\newtheorem*{thmaa*}{Theorem A'}
\newtheorem*{thmb*}{Theorem B}
\newtheorem*{thmo*}{Theorem 1.1}
\newtheorem*{thmc*}{Theorem C}
\newtheorem*{thmd*}{Theorem D}
\newtheorem*{thmf*}{Theorem 4.1}
\newtheorem*{conjecture*}{Conjecture}
\newtheorem*{prop*}{Proposition}
\newtheorem{thm}{Theorem}[section]
\newtheorem{cor}[thm]{Corollary}
\newtheorem{lem}[thm]{Lemma}
\newtheorem{prop}[thm]{Proposition}

\theoremstyle{definition}

\newtheorem*{proof1*}{Proof of (1.5)}
\newtheorem{examp}[thm]{Example}

\newtheorem{remark}[thm]{Remark}

\def\bbq{\mathbb{Q}}
\def\bbf{\mathbb{F}}
\def\bbr{\mathbb{R}}
\def\bba{\mathbb{A}}
\def\bbc{\mathbb{C}}

\def\bbn{\mathbb{N}}

\def\calp{\mathcal{P}}
\def\calg{\mathcal{G}}
\def\calo{\mathcal{O}}

\def\calh{\mathcal{H}}

\def\calz{\mathcal{Z}}
\def\calf{\mathcal{F}}
\def\call{\mathcal{L}}
\def\calt{\mathcal{T}}

\def\a{\alpha}
\def\vare{\varepsilon}
\def\ol{\overline}

\def\suml{\sum\limits}

\def\tle{\triangleleft}
\def\vp{\varphi}

\theoremstyle{remark}  
\newtheorem*{remark*}{Remark}
\begin{document}
\date{}
\title{Normal Subgroup Growth of Linear Groups: \\ the $(G_2, F_4, E_8)$-Theorem}
\author{Michael Larsen and Alexander Lubotzky\\
\ \\
\textsc{\large Dedicated to M. S. Raghunathan}}
\maketitle
\smallskip

Let $\Gamma$ be a finitely generated residually finite group.  Denote by $s_n(\Gamma)$ (resp. $t_n(\Gamma)$)
 the number of subgroups (resp. normal subgroups) of $\Gamma$ of index at most $n$.
In the last two decades  the study of the connection between the
algebraic structure of $\Gamma$ and the growth rate of the sequence $\{ s_n(\Gamma)\}^\infty_{n=1}$
has  become a very active area of research under the rubric ``subgroup growth" (see [L1], [LS] and the references
therein).  The subgroup growth rate of a finitely generated group is bounded above by
$e^{O(n\log n)}$, which is the growth rate for a finitely generated non-abelian free group.
On the other end of the spectrum, the groups with polynomial subgroup growth (PSG-groups for
 short), i.e., those satisfying $s_n(\Gamma)\le n^{O(1)}$, were characterized ([LMS]) as
 the virtually solvable groups of finite rank.  This was originally proved for linear
 groups ([LM]).  The linear case was then used to prove the theorem for general residually
 finite groups.

In recent years, interest has also developed in the {\it normal} subgroup growth
$\{t_n(\Gamma)\}^\infty_{n=1}$. In [L3] it was shown that the normal subgroup growth
of  a non-abelian free group is of type $n^{\log n}$, just a bit faster than polynomial
growth.  One cannot, therefore, expect that the condition on $\Gamma$ of being of
``polynomial normal subgroup growth" (PNSG, for short) will have
the same strong structural implications as that of polynomial subgroup growth.  In particular,
 PNSG-groups (unlike PSG-groups) need not be virtually solvable.  In fact, the examples
 produced in [S, Py]  (which, incidentally, show that essentially every rate of subgroup
 growth   between polynomial and factorial can occur) all have sublinear normal subgroup growth and are very far from being solvable.

For linear groups, however, the situation is quite different.

First fix some notations:  Let $F$ be an (algebraically closed) field and $\Gamma$
a finitely generated subgroup of $GL_n(F)$.  Let $G$
be the Zariski closure of $\Gamma$, $R(G)$  the solvable radical of $G$, and $  G^\circ$-the
 connected component of $G$.  Write $\ol G = G^\circ/R(G) $ and let $S(\Gamma) = \ol G/Z(\ol G)$ --
 ``the semisimple closure of $\Gamma$".
So $S(\Gamma) = \mathop{\Pi}\limits^r_{i=1} S_i$ where each $S_i$ is a simple algebraic group over $F$.
\begin{thma*} Assume $\Gamma$ is of polynomial normal subgroup growth,  and
$S(\Gamma) =\mathop{\Pi}\limits^r_{i=1} S_i$ as above.
Then either

{\rm (a)} $r=0$, in which case $\Gamma$ is virtually solvable, or

{\rm (b)} $r>0$ and for each $1 \le i \le r$, $S_i$ is a simple algebraic group of
type $G_2, F_4$  or $E_8$.
\end{thma*}

Theorem A  is best possible.  Indeed, we will see below (Theorem C) that for $S$-arithmetic
 subgroups of $G_2, F_4$ and $E_8$, the rate of growth of the normal congruence subgroups is polynomial.
 At least some of these arithmetic groups (and  conjecturally all -- see [PR])  satisfy the congruence subgroup property.
 So, they provide examples of Zariski dense subgroups of $G_2, F_4$ and $E_8$ with polynomial normal subgroup growth.
 It should be noted, however, that the normal subgroup growth rate is not determined by the Zariski closure.  Every simple algebraic group has
  a Zariski dense free subgroup and the normal growth of the latter is $n^{\log n}$.

Theorem A is surprising in two ways:  First, it shows that  linear groups are very
different from general residually finite groups;  the linear PNSG-groups are
``generically" virtually solvable. But even more surprising is the special role played
by $G_2, F_4$ and $E_8$.
This seems to be the first known case in which a growth
condition singles out individual simple algebraic groups from all the others.

 What is so special about $G_2, F_4 $ and $E_8$?  These are the only simple algebraic groups whose simply connected
and adjoint forms are the same, or in other words the only groups whose universal covers have trivial center.
Theorem A is therefore equivalent to:
\begin{thmaa*} Let $\Gamma \le GL_n(F)$  and $S(\Gamma) = \mathop{\Pi}\limits^r_{i=1} S_i$ as above.  Assume that for
 some $i, \; 1 \le i \le r, \; \calz(\tilde S_i) \neq 1$ (i.e., the
scheme-theoretic
  center of the simply connected cover of $S_i$
  is non-trivial).  Then $\Gamma $ is not a PNSG-group.
\end{thmaa*}

In fact, our result is much more precise:
\begin{thmb*}  Let $\Gamma \le GL_n(F)$ and $S(\Gamma) = \mathop{\Pi}\limits^r_{i = 1} S_i$ as above.
Denote the characteristic of $F$ by $p\ge 0$.  Then,

{\rm (i)} If for some $i$, \ $1 \le i \le r, \;  \calz(\tilde S_i)\neq 1$ then the normal subgroup growth rate of $\Gamma$
 is at least $n^{\log n/(\log\log n)^2}$.

 {\rm (ii)} If for some $i$, $1 \le i \le r, \; p \big| \; |\calz(\tilde S_i)|$ then the normal subgroup
 growth rate of $\Gamma$ is $n^{\log n}$.
 \end{thmb*}

It is easy to see that Theorem B implies A(and A'), so we will aim at proving the former.  To this end we will
 prove the following result which may be of independent interest:

\begin{thmf*}  Let $A$ be an integral domain, finitely generated over the prime field of characteristic $p\ge 0$, with fraction field $K$.
Let $\Gamma$ be a finitely generated  subgroup of $GL_n(A)$ whose Zariski closure $G$ in $GL_n(K)$ is connected
 and absolutely simple.  Then there exists a global field $k$ and a ring homomorphism $\phi \colon A\to k$, such that the Zariski closure of $\phi(\Gamma)
 $ in $GL_n(k)$ is isomorphic to $G$ over some common field extension of $K$ and $k$.
\end{thmf*}

Theorem 4.1 will enable us to reduce the proof of Theorem B to the case when $\Gamma$ sits within $GL_n(k)$,
 where $k$ is a global field.  Being finitely generated, it is even contained in an $S$-arithmetic group.  The Strong Approximation Theorem for
 linear groups (in the strong version of Pink [P2]) then connects the estimate of $t_n(\Gamma)$ to the counting of  normal
 congruence subgroups in an  $S$-arithmetic group.  Here we can prove the following precise result.
\begin{thmc*} Let $k$ be a global field of characteristic $p\ge 0$, $S$ a non-empty set of valuations of $k$
containing all the archimedean  ones, and $\calo_S=\{x\in k|v(x) \ge 0, \forall v \notin S\}$.  Let
$\calg$ be a smooth group scheme over $\calo_S$ whose generic fiber $\calg_\eta$ is connected and simple .
Let $\tilde \calg_\eta$ be the universal cover of $\calg_\eta$.  Let $\Delta$ be the $S$-arithmetic group $\Delta =
\calg(\calo_S)$.  Assume $\Delta$ is an infinite group.
Let $D_n(\Delta)$ be the number of  normal congruence subgroups of index at most $n$
in $\Delta$.  Then the growth type of $D_n(\Delta)$ is:
%
%

\bigskip

{\rm(i)} \  $n$ if $\calg$ is of type $G_2, F_4$ or $E_8$.

\bigskip

 {\rm(ii)} \ $n^{\log n/(\log\log n)^2}$ if $\calz(\tilde \calg_\eta)\neq 1$ and $p\nmid|\calz(\tilde \calg_\eta)|$
\bigskip

{\rm(iii)} \ $n^{\log n} $ if $p\Big| |\calz (\tilde \calg_\eta)|$

\end{thmc*}

Note that by $|Z(\tilde \calg_\eta)|$ we mean the order of the group scheme which is the center of $\tilde G_\eta$.
This is an invariant which depends only on the root system:  If $p\neq 2, 3, \; p \big| |\calz(\tilde \calg_\eta)|$
 if and only if $\calg$ is of type $A_n$ and $p| (n+1)$.

\bigskip

It is of interest to compare $D_n(\Delta)$ to $C_n(\Delta)$ when $C_n(\Delta)$ counts the number of
all congruence subgroups of $\Delta$ of index at most $n$.
The following table summarizes the situation.

\bigskip

\noindent

\begin{center}
\begin{tabular}{  | c | l | l | }
\hline

&&\\
 & $\qquad p=0$ & $\qquad p>0$  \\
 &&\\ \hline \hline
 & & \\
 &$C\colon \; \;  n^{\log n/\log\log n}\; \; \;[L2, GLP]$ & $C\colon\; \;  n^{\log n} \qquad\qquad \; \;  \! [N]$ \\
$G_2, F_4, E_8$ & & \\
$ $ & $D\colon\; \;  n\qquad \qquad \quad \; \; [LL]$ &$D\colon\; \;  n \qquad \qquad \quad \; \; \; [LL]$ \\
& & \\ \hline
&& \\
 $\calz(\tilde \calg_\eta)\neq 1$ & $C\colon\; \;  n^{\log n/\log\log n}\; \; \;\,  [L2, GLP]$ &
 $C\colon \; \; n^{\log n} \qquad\qquad \quad \!\![N]$\\
&& \\
and $p\nmid|Z(\tilde \calg_\eta)|$ &$D\colon \; \; n^{\log n/(\log\log n)^2} \;\, [LL]$ & $D\colon \; \; n^{\log n/(\log\log n)^2} \; \; \, [LL]$\\
&& \\ \hline
&& \\
 & & $C\colon \; \; n^{\log n} \qquad\qquad \quad \![N] $ \\
$p\big| |\calz(\tilde \calg_\eta)|$ & cannot occur & \\
 & & $D\colon\; \;  n^{\log n}\qquad \qquad \; \; [LL]$\\
 && \\ \hline

\end{tabular}
\end{center}

\bigskip

We only remark, that as of now the results of $[N]$ require the assumption that $\calg$ splits.
[LL] refers to the current paper.
\bigskip

The proof of Theorem C depends on a careful analysis of the corresponding problem over local fields.
Here we have:
\begin{thmd*} Let $k$ be a non-archimedean local field of characteristic $p \ge 0$, and $\calo$
its valuation ring.  Let $\calg$ be a smooth group scheme over $\calo$ whose generic fiber $\calg_\eta$
 is connected, simply connected and simple.  Let $\Delta = \calg(\calo)$ and $t_n(\Delta)$ the number
  of open normal subgroups of $\Delta$ of index at most $n$.  Then the growth type of $t_n (\Delta)$
  is:
  \begin{enumerate}
  \item[\rm (i)] \ $n^{\log n} $ if $ p | \Big| \calz (\calg_\eta)\Big|$
  \item[\rm (ii)]  \ $n$ if $\calg$ is of Ree type, i.e., either $p = 2$ and $\calg$ is
  of type $F_4$ or $p=3$ and $\calg$ is of type $G_2$.
  \item[\rm (iii)] \ $\log n$ otherwise.
  \end{enumerate}
  \end{thmd*}
\bigskip

Theorem D is somewhat surprising in its own right; the groups of Ree type which play a special role appear
 here for an entirely different reason than $G_2, F_4 $ and $E_8$ appear in Theorem C.
In fact, $G_2, F_4$ and $E_8$ appear in Theorem C because their (schematic) center is trivial, while
 the groups of Ree type appear as exceptions because their
adjoint representations are reducible.  In all other cases of reducibility, $p$ divides the order of the center;
in particular, the groups of Suzuki type are of type (i).

 It is also of interest here to compare the growth of $t_n(\Delta)$ to $s_n(\Delta)$,
 the number of all open subgroups of index at most $n$.  For $s_n(\Delta)$  the result is  simple:
$s_n(\Delta)$  grows polynomially if $p=0$ [LM] and as fast as $n^{\log n}$ if $p > 0$ (see [LSh]).

  The paper is organized as follows:  In \S1, we collect some general results on counting normal
  subgroups and other preliminaries.  Special attention is called to Proposition 1.5 which
   seems to be new and useful.  In \S2, we treat the local case and prove a stronger version of Theorem D,
    and in \S3 Theorem C is proven.  Section 4 deals with the question of specializing groups while
preserving their Zariski closures and Theorem 4.1 is proved.  All this is collected to deduce Theorem B in
     \S5.

In preparing for the proof of Theorem B
a subtle difficulty has to be confronted:
 with subgroup growth, one can pass without restricting generality to a finite index subgroup and so one can always
 assume that the Zariski closure $G$ of $\Gamma$ is connected.  On the other hand, normal subgroup
  growth may be sensitive to such a change.
  We must therefore handle also the non-connected case.
So  Theorem C and Theorem D are proven also for the case where $G$ is not necessarily connected.

Raghunathan has made fundamental contributions to the study of congruence subgroups (cf. [R1] [R2] and [R3]).
We are pleased to dedicate this paper, which counts congruence subgroups, as a tribute to him.

 \section*{Notations and conventions}

If $g, f\colon\bbn \to \bbr$ are functions, we say that $g$ grows at least as fast
as $f$ and write $g\succeq f$  if there exists a constant $0 < a \in \bbr$ such that $g(n)\ge f(n)^a$ for every large $n$.
We say that $g$ and $f$ have the same growth type if $g \succeq  f$ and $f \succeq  g$, or equivalently if $\log f(n) \approx \log g (n)$.

Algebraic groups are geometrically reduced, possibly disconnected affine group scheme of finite type over a field.
They are generally  written in italics.
We use calligraphic letters for groups schemes to emphasize that they are {\it schemes}, either because
the base is not a field or because we wish to allow non-reduced groups.
The superscript ${}^\circ$
denotes identity component and $\tilde X$ is the universal covering group of $X$.  Semisimple
groups are connected, but simple groups may have a finite center.

If $\Delta$ is a discrete (resp. profinite) group we denote by $T_n(\Delta)$ the set of normal
(resp. normal open) subgroups of $\Delta$ of index at most $n$ and $t_n(\Delta) = |T_n(\Delta)|$.

\section{Counting normal subgroups and preliminaries}

In this short section we assemble few propositions mainly about counting normal subgroups in finitely
 generated (discrete or profinite) groups, that will be used in the following sections.
\begin{prop}  Let $\Gamma$ be a finitely generated group and $K$ a finite normal subgroup of $\Gamma$. Then
 the normal subgroup growth type of $\Gamma $ is the same as that of $\Delta=\Gamma/K$.
 \end{prop}
\begin{proof}  Clearly $t_n(\Gamma) \ge t_n(\Delta)$.  On the other hand,
every $N\in T_n(\Gamma)$
  gives rise to a subgroup $NK/K$ of $\Delta$ of index at most $n$.
  So the map $N\to N K/K$ is a surjective map from $T_n(\Gamma)$ onto $T_n(\Delta)$.
  If it is not true that  $t_n(\Delta) \succeq t_n(\Gamma)$,
  then for infinitely
    many values of $n$, $t_n(\Delta) \le t_n(\Gamma)^{1/2}$.  For such an $n$, the fiber of at least one element
     in $T_n(\Delta)$ is of order at least $s=\left[t_n(\Gamma)^{1/2}\right]$,
     i.e., there exist $N_1, \dots, N_s \in T_n (\Gamma)$
     such that all $N_iK$ are equal to each
     other -- say to $\ol N$.  There are only a bounded number $c_1$ of possibilities
     for $N_i\cap K \tle K$, hence by replacing $s$ by $\left[\frac{s}{c_1}\right]$
we can assume
that all the groups  $N_i$  have the same intersection $K_1$ with $K$.
Thus $\ol N/K_1\simeq N_i/K_1\times K/K_1$ for every $i$.
As there are $\left[\frac{s}{c_1}\right]$ different
 $N_i/K_1$, it follows that $\ol N/K_1$ has at least
 $\left[\frac{s}{c_1}\right]$ different $\Gamma$-homomorphisms to $K/K_1$.
 The number of endomorphisms of
 $K/K_1$  is bounded by $c_2$, so $N_i/K_1\simeq \ol N/K$ has at least
 $\left[\frac{s}{c_1c_2}\right]$\
 different $\Gamma$-homomorphisms to $K/K_1$.  This shows that $\ol N/K$ has at least
 $\left[\frac{s}{c_1c_2c_3}\right]$
 subgroups which are normal in $\Gamma/K$ and their index in $\ol N$ is at most $|K/K_1|$, so their
 index in $\Gamma$ is at most $[\Gamma\colon \ol N]\cdot |K/K_1| \le \frac{n}{|K/K_1|} \cdot |K/K_1|=n$.
 Thus $t_n(\Delta) \ge \left[\frac{s}{c_1c_2c_3}\right]$ and $\Delta$ has the same growth as of $\Gamma$.
 \end{proof}
 \begin{prop}  Let $\Gamma$ be a finitely generated group and $\Delta $ a finite index normal subgroup of $\Gamma$.
 Then there exists two constants $c_1$ and $c_2$ such that
 \begin{equation*}  t_n(\Gamma) \le c_1n^{c_2} t_n(\Delta)
 \end{equation*}
 \end{prop}
 \begin{remark*}  We do not know a useful upper bound on $t_n(\Delta)$ in terms  of $t_n(\Gamma)$.
 Such a bound could save us the trouble of   treating non-connected algebraic groups.
 \end{remark*}
\begin{proof}  If $N \in T_n(\Gamma)$  then $D=N\cap\Delta \in T_n(\Delta)$  and $K=N\Delta$ is normal in $\Gamma $ containing $\Delta$.
 Given $D$ and $K$ as above, the number of $N\tle \Gamma$ with $N\cap \Delta = D$ and $N\cdot
 \triangle = K$ is bounded by  $n^c$ for a suitable constant $c$.  Indeed, $N/D$ should be a normal
 complement to the normal subgroup $\Delta/D$ in $K/D$.  The number of such complements is bounded by the number of
possible $\Gamma -$ homomorphisms from $K/D$ to $\Delta/D$.  The latter is at most $|\Delta /D|^{d_\Gamma (K)}$
 where $d_\Gamma (K)$ denotes the number of $\Gamma$-generators of $K$.  The proposition now follows with $c_1$
 equals the number of normal subgroups $K$ of $\Gamma$ containing $\Delta$ and $c_2$ the maximum of $d_\Gamma (K)$
 over these  possible $K$.
 \end{proof}
 \begin{lem} Let $G = A \times B$ be a product of two groups and $N\tle G$.  Let $A_1=N\cap A,
  B_1=N\cap B, A_2 = \pi_A(N)$ and $B_2 = \pi_B(N)$ where $\pi_A$  and $\pi_B$ are the projections
   to $A$ and $B$, respectively.  Then
   \begin{enumerate}
   \item[\rm (i)] $A_2/A_1$ (resp. $B_2/B_1)$ is central in $A/A_1$ (resp. $B/B_1$).

   \item[\rm (ii)]  There exists an isomorphism $\vp\colon A_2/A_1 \to B_2/B_1$ such
    that $N/(A_1\times B_1)$ is the graph of $\vp$.
    \end{enumerate}
    \end{lem}
    \begin{proof}   Clearly $A_2/A_1\tilde\leftarrow N/(A_1\times B_1) \tilde\rightarrow B_2/B_1$,
and this defines $\vp$ satisfying (ii).  We claim that $A_1/A_2$ is central in $A/A_1$. Indeed,
for all $x \in A$
\begin{equation*}
(xA_1, B_1)^{-1} \left(aA_1, \vp (aA_1)\right) (xA_1, B_1)= (x^{-1}axA_1, \vp(aA_1)\big).
\end{equation*}
As $\vp$ is an isomorphism, this implies $x^{-1} axA_1 = aA_1$, i.e., $aA_1$ commutes with $xA_1$.
By symmetry $B_2/B_1$ is central in $B/B_1$ and (i) is proved.
\end{proof}
\begin{cor} Let $G=A\times B$ as above.  If for every (finite index) normal subgroup $M\tle A, Z(A/M) = \{ 1\}$, then
 every  (finite index) normal subgroup $N\tle G$ is of the form $N=(N\cap A) \times
  (N\cap B)$.
  \hfill $\square $
\end{cor}
\begin{prop}  Let $G=A\times B$ be a product of two groups.  Then:
\begin{equation*}
t_n(G)\le t_n(A)^2\cdot t_n(B)^2\cdot z_n(A)^{\delta_n(A)}
\end{equation*}
where
\begin{equation*}
z_n(A) = \max \{ | Z(A/N)| \Big| N\tle A \; \; \text{and} \; \; [A\colon N]\le n\}
\end{equation*}
 and
 \begin{equation*}
\delta_n(A) = \max
\{ d\big(Z(A/N)\big) | N\tle A \; \; \text{and}\; \;  [A\colon N]\le n\}
\end{equation*}
(when $d(X)$ denotes the number of generators
of $X$).
\end{prop}
\begin{proof}
Apply Lemma 1.3.
The two pairs of normal subgroups $A_1 \le A_2$ in $A$ and $B_1 \le B_2 $ in $B$ (so that $A_2/A_1$
 (resp. $B_2/B_1$) is central in $A/A_1$ (resp. $B/B_1)$) together with the isomorphism
  $\vp\colon A_2/A_1\to B_2/B_1$ determine $N$.
  This proves that
  \begin{equation*}
  t_n(G)\le t_n(A)^2\cdot t_n(B)^2\cdot h_n
  \end{equation*}
  where $h_n$ is the maximum possible number of isomorphisms from $A_2/A_1$ to $B_2/B_1$ as above,
  or equivalently, $h_n$ is an upper bound on the number of automorphisms of $A_2/A_1$ when $A_1\le A_2$
   are normal subgroups of index at most $n$ in $A$, such that $A_2/A_1$ is central in $A/A_1$.
   Clearly $h_n \le z_n(A)^{\delta_n(A)}$ (note that as $A_2/A_1$ is abelian and central
   in $A/A_1$, $d(A_2/A_1)\le d\big(Z(A/A_1)\big)$.  The proposition is therefore proved.
   \end{proof}
\begin{prop} If $\Gamma$ is  a finitely generated discrete or profinite group, then
$t_n(\Gamma) \preceq n^{\log n}$.
\end{prop}
\begin{proof}
This is proved in [L3] for the free groups; it therefore follows for every group.
\end{proof}
\begin{remark}  While  the proof of the general result in [L3] requires the classification of the finite
 simple groups (CFSG), this is not always needed for a given profinite or discrete group.
 The CFSG has been used in [L3] via the result of Holt [Ho] which implies that for every finite simple
  group $G$, every prime $p$ and every simple $\bbf_p[G]$-module $M,  \dim H^2(G, M) = O(\log|G|) \dim M$.
  Now, if $\Gamma$ is a profinite group whose finite composition factors satisfy Holt's inequality  (for every
   $p$ and every $M$) then  Proposition 1.6 holds for $\Gamma$.  Now, the proof of Holt in [Ho] for the {\it known} simple
   groups is still valid, even if one does not assume the CFSG.  In our papers, all the relevant
    profinite groups are such that almost all their composition factors are known, so Proposition 1.6 holds for them.
It is worth mentioning that  we also  use  [P2] later, which not only improves [W]
 but also frees it from CFSG.
    Our paper is therefore classification free!
    \end{remark}
\begin{prop} Let $G$ be a simple algebraic group defined over an algebraically closed field $F$ of characteristic
 $p \ge 0$.  Let $\tilde G$ be the universal cover of $G$ and $\calz(\tilde G)$ its scheme-theoretic
 center.  Assume $p\nmid |\calz(\tilde G)|$ and that the action of $G$ on $L=Lie(G)$ is not irreducible.
Then  either \ (i)\  $p=2$ and $G$ is of type $F_4$ or \ (ii) \ $p=3$ and $G$ is of type
$G_2$.
In either case $L$ has an ideal $I \tle L$ such that $L/I$ is isomorphic to $I$ (as a Lie algebra and as $G$-module).
In case \ (i), $I$ is of type $D_4$   and in case (ii) of type $A_2$.
\end{prop}
\begin{proof}  See [H].
\end{proof}
 Throughout the paper if we are in either case (i) or (ii), we say that $G$ is a group of {\bf Ree type}.

\section{The local case}

Let $k$ be a local (non-archimedean) field of characteristic $p \ge 0$, and $G$ an algebraic group defined
 over $k$, with a semisimple connected component $G^\circ$, whose universal cover we denote $\widetilde{G^\circ}$.
 Let $\calz$ be the scheme-theoretic center of $\widetilde{G^\circ}$.  It is a finite group scheme of order
 $z = |\calz |$.  In other words $z$ is the dimension of the coordinate ring of $\calz$
  as a vector space over $k$.  Another way to think about $z$ is as the index of the lattice generated by the absolute roots of $G^\circ$ in the
   lattice of weights.

    We can now state the main result of this section.
\begin{thm}  Let $G$ be as above and $M$ a (topologically)      finitely generated Zariski-dense
compact subgroup of $G(k)$.  Then the normal subgroup growth rate of $M$ is:
\begin{enumerate}
\item[\rm (i)] $n^{\log n} $ if $p \Big| \, |\calz|$
\item[\rm (ii)] $n$ if $p\nmid |\calz|$ and $G^\circ$ has a simple factor of Ree type
\item[\rm (iii)] $\log n$ otherwise.
\end{enumerate}
\end{thm}
Let us start with an example which is also treated in [BG].
\begin{examp}  Let $M=SL_d(\bbf_p[[t]])$.  If $p\nmid d$ then for every open normal subgroup
$N$ of $M$, there exists $r \in  \bbn$ such that $Q_r \subseteq N \subseteq Z_r$ where
\begin{equation*}
Q_r =Ker
(SL_d(\bbf_p[[t]])\to SL_d(\bbf_p[[t]]/(t^r))
\end{equation*}
 and $Z_r$ is the preimage in $M$ of the center of $\ol M(r) = SL_d
(\bbf_p[[t]]/(t^r))$.  Such $N$ is of index approximately $p^{(d^2-1)r}$ and
\ $ Z_r/Q_r$ is of order bounded independently of $r$, so there are only a bounded number of possibilities for such $N$.  Hence $t_n(M)$ grows like $\log n$.

On the other hand, if $p|d$, then for every $r$, the group $\ol M(pr)$ has a large center:
it consists of all the scalar matrices of the form $(1+y) I_d$ where $y \in (t^r)/(t^{pr})$.  Note that
$(1+y)^p=1$ in the ring $\bbf_p[[t]]/(t^{pr})$, so $det\big((1+y)I_d\big) = 1$.
Now, $|(t^r)/ (t^{pr}) | = p^{(p-1)r}$ and so $\ol M(pr)$ has a $p$-elementary abelian central subgroup
 of rank $(p-1)r$.  Hence, it has at least $p^{\frac14(p-1)^2r^2}$ normal subgroups of index at most
 $|M(pr)| \approx p^{(d^2-1)pr}$.
 Therefore $t_n(M)$ grows at least as fast as $n^{\log n}$. By (1.6),  this is the largest normal subgroup
 growth possible.
\end{examp}
 The proof of Theorem 2.1  will depend on a careful analysis of principal congruence subgroups.

 Let  $\calo$ be the  discrete valuation ring of $k$,\  $\pi$ a uniformizer and $\bbf$
  the residue field.  Let $\calg/\calo$ denote a smooth group scheme.  In particular, if $G$
  is a semisimple algebraic group over $k$ and $X$ the associated Bruhat-Tits building, then for every
   point $x$ of $X$, the stabilizer of $x$ in $G(k)$ is equal to $\calg(\calo)$ for some such $\calg$
[BT, 5.1.9].  For each positive integer $n$, \ $\calg(\calo/\pi^n\calo)$ is finite and the reduction map
$\calg(\calo) \to\calg(\calo/\pi^n\calo)$ is surjective (since $\calg$ is smooth).
The kernel
 $Q_n$ is called the  $n$-th {\it principal congruence subgroup}.  These subgroups are closely related to the Lie
  algebra $\call=Lie (\calg/\calo)$, which is by definition the dual of the pull-back of the relative cotangent
   bundle $\Omega^1_{\calg/\calo}$ by the identity section.
   As $\calo$-module, $\call$ is free of rank equal to the relative
   dimension of $\calg/\calo$. This construction commutes with base change [SGA3, II 4.11]; in particular, $\call\otimes k$
    and $\call\otimes \bbf$ are the Lie algebras of the generic and special fiber, respectively.
In the case that $\calg$ is Chevalley, i.e. split semisimple, and $\calo \simeq\bbf[[\pi]]$ then $\call$ is isomorphic to $(\call\otimes\bbf)[[\pi]]$.
If $a, b, \in \bbn$ with $a\le b\le 2a$, there exists a canonical isomorphism
\begin{equation*}
\log_{(\pi^a)/(\pi^b)}\colon Q_a/Q_b\to (\pi^a\calo/\pi^b\calo)\otimes \call
\end{equation*}
(see [P1, 6.2]).
If $c, d \in \bbn$ and $c\le d \le 2c$ the square
$$
\CD
Q_a\big/Q_b\times Q_c\big/Q_d  @>[\ , \ ]>> Q_{a+c}\big/Q_{a+d}\cdot Q_{b+c} \\
@V\log VV @ VV\log V \\
\pi^aL\big/\pi^bL\times \pi^cL\big/\pi^d L @> [\ , \ ]>>  \pi^{a+c}L\big/(\pi^{a+d}L+\pi^{b+c} L).\\
\endCD
$$
commutes.

We can  now begin the proof of Theorem 2.1, starting with  the lower bounds.  One can easily see that
 the principal congruence subgroups (with respect to any fixed faithful representation of $G$
 into $GL_n$)  described above assure  that the normal subgroup growth is always
  at least logarithmic.

The cases of interest for lower bounds are when either $p\big|\,|\calz|$ or $G$ is of  Ree type.

The reader should note that  these lower bounds arguments are complicated by the difficulty
 mentioned in the introduction (see also (1.2)) that obliges us to consider non-connected groups.

We will use the notation and terminology of [P1].  Replacing $G$ by a quotient we may assume that $G^\circ$
  is a product of isomorphic adjoint simple groups $G_i, i = 1, \dots, r$, that $G/G^\circ$ acts transitively on  the factors $G_i$,
   and that $Cent_G(G^\circ)$ is trivial.  Let $F=k^r$,  $M^\circ=M\cap G^\circ(k)$ and $\calg^\circ$ be the adjoint
   simple group scheme over $F$ such that $\calg^\circ (F)=G(k)$.
   Let $(E, \calh^\circ, \vp)$ be the  minimal
   quasi-model of $(F, \calg^\circ, M^\circ)$.  Thus $\calh^\circ$ is adjoint simple, $E$ is a closed subalgebra of $F$,
    which is a product of local fields,  $M^\circ$ can be regarded as a Zariski dense subgroup of $\calh^\circ(E)$
and $\vp:\calh^\circ\underset{E}{\times} F\to \calg^\circ$     is an isogeny, which induces an isomorphism
on the embeddings of $M^\circ$.  By the (essential) uniqueness of the minimal quasi-model
[P1, 3.6] and the transitivity
of the action of $M/M^\circ$ on the factors of $F$, we conclude that $M/M^\circ$ acts transitively
on the factors of $E$.  Furthermore, $E$ has the same number of factors as $F$ since $M^\circ$
 is Zariski dense in $G^\circ$.  So we write $E={k'}^r$ where $k'$ is a (local) subfield of $k$.  The restriction of scalars of $\calh^\circ$
  to $k'$ is a product $H^\circ$ of isomorphic adjoint simple factors $H_i, \; i = 1, \dots, r$ and $\vp$ is a
  product of identical isogenies $\vp_i\colon H_i\underset{k'}{\times} k\to G_i$.
Let $H^{out}=H^\circ \rtimes Out (H^\circ), \, \widetilde{H^{out}}=\widetilde{H^\circ} \rtimes
 Out (H^0)$, and $G^{out} = G^\circ\rtimes Out (G^\circ)$.  Note that $Out(G^\circ)$ and $Out (H^\circ)$ are canonically
  isomorphic.  In fact $G^\circ$ and $H^\circ$ have the same root system except possibly in characteristic
   2 where there exist isogenies between groups of type $B_n$ and $C_n$.  Even though $G$ need not map to $G^{out}$,
    if $\ol k$ is an algebraic closure of $k$, $G(\ol k)$ is naturally a finite index subgroup of
 $G^{out}(\ol k)$, and we regard $M$ as a subgroup of the latter.
We extend $\vp$ to $\vp^{out}\colon H^{out}\underset{k'}{\times}
k \to G^{out}$ to identify $G^{out} (\ol k)$ with $H^{out} (\ol k)$.
Finally, we pull back by the natural map
$\widetilde{H^{out}}\to H^{out}$ to obtain a group $\tilde M \le \widetilde{H^{out}}
(\ol k)$.  As $M$ is a quotient of $\tilde M$ by a finite normal subgroup, it suffices by (1.1) to give a lower bound for $t_n(\tilde M)$.

Let $\tilde M^\circ = \tilde M \cap \tilde H^\circ (k')$.  By [P1, 0.2], $\tilde M^\circ$
is an open subgroup of $\tilde H^\circ(k')$.  Let $\pi$ denote the universal cover map $\tilde H^\circ \to H^\circ$.
Then $H^\circ (k') / \pi\big(\tilde H^\circ (k')\big)$ is an abelian torsion group.
As $\tilde M$ is finitely generated, the image of $\tilde M\cap \tilde H^\circ(\ol k)$
in this quotient is finite, so $\tilde M^\circ$ is a finite index subgroup of $\tilde M\cap \tilde H^\circ (\ol k)$
and therefore
 of $\tilde M^\circ$.  As $M$ normalizes $M^\circ$, a Zariski dense subgroup of
 $H^\circ$, it normalizes $H^\circ$ itself and likewise $H^\circ(k')$; by the same argument $\tilde M$
 normalizes $\tilde H^\circ(k')$.  We conclude that $\tilde M^\circ$ is normal in $\tilde M$.

As $\tilde M$ is compact, its conjugation
 action fixes a point $x$ in the Bruhat-Tits building of $\widetilde{H^\circ}(k')$.
 Let $\calf$ be the smooth group scheme over $\calo$ corresponding to $x$.  As $\tilde M$ fixes $x$, conjugation by
  any element of $\tilde M$ gives an automorphism of $\calf$.  Let $Q_n$ denote the $n$-th principal congruence subgroup
   of $\calf(\calo)$.    By construction, it is normalized by $\tilde M$.

    Let $\calz$ denote the identity component of the scheme theoretic center of $\widetilde{H^\circ}$ and $\calz_\calf$
     the Zariski closure of $\calz$ in $\calf$.  Note that for any $\calo$-algebra $R$, $\calz_\calf(R)$ lies
      in the center of $\calf(R)$.
\begin{lem} Let $\calo=\bbf[[\pi]], \; \Delta$ a finite group and $B$ a commutative $\calo[\Delta]$-algebra which
 is finite as a module.
Suppose that the nil radial $I$ of $B\underset{\calo}{\otimes} k$ is non-trivial and $\dim_k(B\underset{\calo}{\otimes}k)/I =1$.
Then there exists a non-trivial $\bbf[\Delta]$-module $T$ and a positive integer $\gamma$
 such that for every $n\in\bbn$, $Hom_{ring}(B,\calo/\pi^{2\gamma n}\calo)$ contains $T^n$ as a submodule.
\end{lem}

We apply the lemma in the case B is the coordinate ring of $\calz_\calf$ and $\Delta=\tilde M/(\widetilde{H^\circ}(k')\cap\tilde M)$;
 note that $Hom_{ring}
 (B, \calo/\pi^{2\gamma n}\calo)$ and hence $T^n$ sits as a subgroup of
 $\tilde M/Q_{2\gamma n}$ for $n$
 sufficiently large.  Now, as $T^n\simeq T\otimes (\calo/(\pi))^n$ as $\Delta$-modules, we can deduce
  that $\tilde M$ has at least $p^{\frac14 n^2}$ normal subgroups of index at most
  $\big|\tilde M/Q_{2\gamma n}\big|\sim
  c_0^{2\gamma n}$ for some constant $ c_0$.  This finishes the proof of the lower bound modulo the lemma.
  \vskip.2cm
\noindent
{\it Proof of Lemma 2.3}  Let $J$ be the nil radical of $B$ and $S=\bigcup\limits^\infty_{r=1} \{ x \in B\big| \pi^r x \in J\}$
 its $\pi$-saturation.  Then $B/S$ is a finitely generated torsion free $\calo$-module, hence free
  over $\calo$, and the codimension condition on $I=S\underset{\calo}{\otimes} k$
  implies that the rank of $B/S$ is  $1$.
  Thus $B=\calo\oplus S$ as $\calo$-module.  Any $\calo$-linear map from $S/S^2$
  to $\pi^{\gamma n}\calo/\pi^{2\gamma n}\calo$ gives
an $\calo$-algebra homomorphism $B\to \calo/\pi^{2\gamma n}\calo$.  By Nakayama's Lemma,
$(S/S^2)\underset{\calo}{\otimes} k = I/I^2\neq (0)$, so the free $\calo$-module $(S/S^2)/(S/S^2)_{tor}$ is a non-zero
$\calo[\Delta]$-module and $N=Hom_\calo(S/S^2, \calo)
$ is a non-zero
$\calo[\Delta]$-module
which is
$\calo$-free.
It suffices to find an $\bbf[\Delta]$-module appearing with multiplicity at least $n$ in $N/\pi^{\gamma n}N$.
Choose $v \in N\setminus \pi N$ and then choose
 $\gamma$ such that $\bbf[\Delta]v\setminus \{ 0 \} \subseteq N\setminus \pi^\gamma N$.  Then
\begin{equation*}
\bbf[\Delta]\{v,\pi^\gamma v, \dots, \pi^{\gamma(n-1)}v\}\simeq \mathop{\oplus}^{n-1}_{i=0} \bbf[\Delta]
 \pi^{i\gamma}v\simeq\mathop{\oplus}^{n-1}_{i=0} \bbf[\Delta]v.
 \end{equation*}
 This proves the lemma with $T=\bbf[\Delta]v$.
 \hfill $\square$
\vskip.2cm

 We turn now to proving the lower bound for the Ree cases.
 \begin{lem}  Assume $char(k) = 2 $ (resp. 3) and $G^\circ$ has at least one factor of type $F_4$ (resp. $G_2$).
 If $M$ is a compact open subgroup of $G(k)$ then $t_n(M)$ has growth rate at least $n$.
 \end{lem}
 \begin{proof}  Without loss of generality we may assume
that $G^\circ$ is isotypic and $M/M\cap G^\circ (k)$ acts transitively on the factors.  We claim that
every $k$-automorphism of a simple group $H$ which is both adjoint and simply connected is an inner automorphism
 by an element of $H(k)$.
 Indeed,  Dynkin diagrams of  adjoint simply connected simple groups
 have no symmetries, so every
  automorphism is of the form $ad(x) $ for $x \in H(\ol k)$;
  as $ad(x)$ is a $k$-automorphism of the Lie
  algebra of $H$ and the adjoint representation
   is faithful, this implies
   $x \in H(k)\subset GL_{\dim (H)}(k)$.
   We can therefore write $G=G^\circ \rtimes M/M^\circ$.

 As before we may assume that the image $\Delta$ of $M$ in $Out (G^\circ)$ acts transitively
  on the factors.  The projection of $M\cap G^\circ(k)$ to each factor is then the same: a compact open
  subgroup $C$ of $F_4(k)$ (resp. $G_2 (k)$).
  Let $\calh$ denote a smooth group scheme  over $\calo$ such that $C$ lies in $\calh (\calo)$ as an open subgroup.
  Thus $M$ is contained in $\calh(\calo)^r\rtimes G/G^\circ$ and by Propositions 1.1 and 1.2, we can assume
   $M=\calh(\calo)^r\rtimes \Delta$.
   If $N\in T_n \big(\calh(\calo)\big)$,  then $N^r$ is an open
    normal subgroup of $M$ of index $\le | \Delta |n^r$, so without loss of generality we may assume that $G=G^\circ$ is simple and $M=\calh(\calo)$.
    Let $\call=Lie
    (\calh/\calo)$.  By [1.8], $\call\otimes k$ has a unique proper ideal $I$,
     and $(\call\otimes k)/I$ is isomorphic as $\call\otimes k$ module to $I$.  We fix an isomorphism
     $\psi\colon (\call\otimes k)/I \to I$ such that $\psi(\call/I\cap \call)\subseteq \call$.

     Let $d$ be a positive integer and $q(x)$ a polynomial of degree less than $d$ with coefficients in the
      field of constants of $k$.  Let
\begin{equation*}
N_{d, q}=\{\pi^{2d}\psi(\a) +\pi^{3d}q(\pi)\a\big|\a\in \call/\pi^{2d}\call\}\subseteq
\pi^{2d}\call/\pi^{4d}\call
\end{equation*}
For fixed $d$ and varying $q$, we obtain a family of pairwise distinct $\call$-submodules of $\pi^{2d}\call/\pi^{4d}\call$ of cardinality
 exponential in $d$.  Pulling back by the map $\log_{(\pi^{2d})/(\pi^{4d})}$ we obtain
  subgroups of $Q_{2d}/Q_{4d}$ and therefore open subgroups of $M$.  By [P1, 6.2(c)] these subgroups
  are normal. So we have exhibited $c^d_1$ normal open subgroups of $M$ of index at most
  $[M\colon Q_{4d}]\sim c^d_2$.  This finishes the proof of the proposition. We have therefore
  proved the lower bounds in all cases of Theorem 2.1.
\end{proof}

We turn now to the proof of the upper bounds.  First note
that the upper bound for case (i) follows immediately from (1.6), as $n^{\log n}$ is the maximal normal
subgroup growth rate possible for finitely generated groups.

We now turn to the proof of the upper bound in the remaining cases.  We have already seen that
without loss of generality we may assume that $M$ contains an open subgroup which is also an open
 subgroup of the $k$-points of a connected, simply connected semisimple group.  For upper bounds, we are free to pass to
 subgroups of finite index (see 1.2)), so we can assume from now on that $M$ is open in $G(k)$.  Let us
 start with the generic case.
\begin{prop}
Let $G/k$ be a simple algebraic group not of Ree type and for which $p\nmid |\calz(\tilde G)|$.
Let $M\le G(k)$ be a compact open subgroup.  Then

\bigskip

{\rm (i)} \ there exists a constant $c$ such that for every
$N$ open and normal in $M$, $|Z(M/N)|< c$;

\bigskip

{\rm (ii)} \  $t_n(M) \preceq\log n$.
\end{prop}
\begin{proof}  Let $\calg/\calo$ be a smooth group scheme with generic fiber $G$ and $M \subseteq
\calg(\calo)$.  Let $\call=Lie(\calg/\calo)$.  By (1.8), $L = \call\otimes k$ is a simple Lie algebra,
 so there exists $\ell \in \bbn$ such that the $\ell$-th iterated Lie bracket $[L, [L, [\dots, [L, x]]=L$
  for all non-zero $x\in L$.  Thus $[\call, [\call, [\dots, [\call, x]]$ contains an open neighborhood
  $\pi^{n_x}\call$ of zero for all $x \in \call\setminus \{ 0 \}$.    By compactness, there is a
  uniform upper bound $n$ on $n_x$ as $x$ ranges over $\call\setminus \pi\call$.

  Replacing $M$ if necessary by a finite index subgroup, we may assume $M = Q_m$, for some $m>n$.
  If $\bbn \ni r\ge m$ and $x\in \pi^r\call\setminus \pi^{r+1} \call$  then for every $i \in \bbn$,
  \begin{equation*}
  [\pi^i\calt,[\calt, \dots, [\calt, \ol x]\dots] \supseteq\pi^{i+\ell m+r+n}\call/\pi^{(i+\ell+1)m+r}\call
  \end{equation*}
where $\ol x$ is the projection of $x$ to $\pi^r\call / \pi^{r+m} \call$, $ \calt = \pi^m\call/\pi^{2m}\call$ and the bracket
is taken $\ell $ times.  (We are using the diagram, introduced earlier, computing
 brackets of quotients of principal congruence subgroups).
Thus the topological normal closure of every element $y \in Q_r\setminus Q_{r+1}$ contains representatives of every class
 in $Q_{i+(\ell+1)m+r-1}\big/Q_{i+(\ell+1)m+r}$.
 So this normal subgroup contains $Q_{(\ell + 1) m+r-1}$.

 This shows that if $N$ is a normal open subgroup of $M$ containing  an element outside $Q_{r+1}$, then $N$
  contains $Q_{(\ell + 1) m+r-1}$.  As $\ell$ and $m$ are constants,  this shows that there exists a constant $c'$ such that for every normal
  subgroup $N$ of $M$, there exists $n \in \bbn, \; n = O(\log [M\colon N])$, such that
  $Q_n\subseteq N\subseteq Q_{n-c'}$.  The order of $Q_{n-c'}/Q_n$ is bounded, so for every $n$, there
   are only finitely many  such possibilities for  $N$.  This proves that the normal subgroup growth of $M$ is at most
   logarithmic.    We also see that the center of $Q_m/N$ is included
   in $Q_{n-c'-m}/N$, so its order is bounded.
\end{proof}

Before passing to the semisimple case,  let us first consider the Ree case for simple groups.  We can then treat the semisimple
case uniformly.
\begin{prop}  Let $G/k$ be a simple group of Ree type and $M\subset G(k)$ an open compact subgroup.
Then

{\rm (i)} \ There exists a constant $c$ such that for every open normal subgroup $N$ of $M$, $|Z(M/N)| <c$.

{\rm (ii)} \  $t_n(M) \preccurlyeq n$.

{\rm (iii)} \ There exists a universal constant $\gamma$, independent of $k$, such that if $M$
 is hyperspecial, then for every normal subgroup $N$ of $M$, $Z(M/N)=\{1\} $ and  $t_n(M) \le n^\gamma$ for every $n$.
\vskip.2cm
\noindent
\end{prop}
{\it Proof.}  By [Ti], as $k$ is local non-archimedean, $G$ splits over $k$.  There is therefore
 a split simple group scheme $\calg/\calo$ with generic fiber $G$, so that
  $M$ is commensurable to $\calg(\calo)\subseteq \calg(k) = G(k)$.  By
 Proposition 1.2  we can replace $M$ by any of its  open subgroups, in particular by a principal
  congruence subgroup $Q_m$ of $\calg(\calo)$ when $m$ is sufficiently large.  Fix such an  $m$ and take
$M=Q_m$.  If $m$ is hyperspecial, we take $m = 0$.

Let $\call=Lie (\calg/\calo), \, \ol L = \call \otimes \bbf$ and so $\call = \ol L [[\pi]]$.
Let $I$ denote the unique non-trivial ideal of $\ol L$.  For every normal subgroup $N$ of $M$,
 and for $r \in \bbn$, we denote by  $gr_rN$  the quotient $(Q_r\cap N)/(Q_{r+1} \cap N) \subseteq \ol L$.
 Starting with a subset $Y \subseteq  \ol L$, the iterative process $Y\mapsto [\ol L, Y]$
stabilizes after a bounded number $\ell$ of steps to $C(Y)$, the minimal ideal containing $Y$,
i.e., either $0, I$ or $\ol L$.  For all $r\in \bbn$ and $t \ge r + \ell m$
\begin{equation*}
gr_tN\ge C (gr_rN).
\end{equation*}
Thus there exists integers $a$ and $b$, $m \le a \le b$, such that
\begin{equation}\tag{$\ast$}
gr_rN= \begin{cases}
0 &\text{if} \; \,  r < a-\ell m\\
I &\text{if} \;  a\le r < b\\
\ol L &\text{if} \;  r\ge b + \ell m.
\end{cases}
\end{equation}

One can also easily see that $a$ and $b$ or $O(\log [M\colon N])$.
 At this point part (i) follows by the same argument as (i) of Proposition 2.5.
In the hyperspecial case, moreover, the center $M/N$ is trivial.

 It follows that
\begin{displaymath}
gr_r (Q_a\cap Q_b N)=\left\{ \begin{array}{ll}
0 &\textrm{if $r < a$}\\
I &\textrm{if $a \le r < b$}\\
\ol L &\textrm{if $r\ge b$}.
\end{array}\right.
\end{displaymath}

We claim that for a fixed $N$ satisfying $(\ast)$, the set of open normal subgroups $N'\subset M$ satisfying
 $(\ast)$ for the same constant $a$ and $b$, with $Q_a\cap Q_bN' = Q_a\cap Q_bN$ is bounded by a polynomial
  in $|M/N|$.  Indeed, $N' \supseteq Q_{b+\ell m}$ and $Q_bN'/Q_{b+\ell m}$ is generated
   by $(Q_a\cap Q_b N)/Q_{b+\ell m}$ and a bounded number of other elements of $Q_{a-\ell m}/Q_{b+\ell m}$, so the
   number of possibilities  for  $Q_bN'$ is bounded by $p^{c'b}$, hence polynomially in $|M/N|$.  Now, as $N'\tle Q_bN'$ is of
   bounded index, $N'/Q_{b+\ell m}$ is the kernel of a homomorphism from $Q_bN'/Q_{b+\ell m}$ to a group of bounded order.  The number
of generators of a group is logarithmic in its order, so the number of such homomorphism is again bounded
 polynomially  in $|M/N|$.  If $m = 0, N = Q_a
 \cap Q_b N$, so the number of possibilities for $N'$ is 1, i.e., $N' = N$.

 It therefore remains to fix $(a, b)$ and count the number of normal open subgroups $N$ of $M$
with


\begin{equation}\tag{$\ast\ast$}
gr_rN=\begin{cases}
0 &\text{if} \; \, r<a\\
I &\text{if} \; \, a \le r < b\\
\ol L &\text{if} \; \, r \ge b.
\end{cases}
\end{equation}

At this point, if $M$ is hyperspecial, it will be more convenient to reset  $m$ to  $1$.
We now prove by induction that the number of possibilities for $N$ satisfying $(\ast\ast)$ is bounded
above by $|\bbf |^{c_1(b-a)m}$ where $c_1$ is an absolute constant.
For $b-a < 2m$ the claim is trivial.  Suppose $b-a\ge 2m$.  Let $N_1$ and $N_2$ denote two groups in
 this collection.  Now, for $i \in \{ 1, 2\}$,
 \begin{equation*}
 [Q_m, Q_{b-m} N_i] = [Q_m, Q_a \cap Q_{b-m} N_i]\subseteq Q_{a+m} \cap Q_bN_i = Q_{a+m} \cap N_i
 \end{equation*}
 At the associated graded level
\begin{displaymath}
gr_r[Q_m, Q_{b-m}N_i]=gr_r(Q_{a+m}\cap N_i) =\left\{\begin{array}{ll}
0 &\textrm{if $r<a+m$}\\
I &\textrm{if $a +m\le r < b$}\\
\ol L &\textrm{if $r \ge b$}.
\end{array}\right.
\end{displaymath}
So $[Q_m, Q_{b-m}N_i] = Q_{a+m} \cap N_i$.
By the induction hypothesis, the number of possibilities for $Q_{b-m}N_i$
is at most $|\bbf|^{c_1(b-a-m)m}$.  We fix one, so
\begin{equation*}
Q_{b-m} N_1= Q_{b-m} N_2
\end{equation*}
and
\begin{equation*}
Q_{a+m} \cap N_1 = Q_{a+m}\cap N_2.
\end{equation*}
We have a commutative diagram

$$
\CD
@.0 @. 0 \\
@. @VVV  @VVV \\
@.  Q_{b-m}/Q_b @ = Q_{b-m}/Q_b \\
@.  @VVV @VVV \\
0 @>>>  Q_{a+m}/Q_b @>>> Q_a/Q_b @>>> Q_a/Q_{a+m} @>>> 0\\
@. @VVV  @VVV \Big\| \\
0 @>>> Q_{a+m}/Q_{b-m} @>>> Q_a/Q_{b-m} @>>> Q_a/Q_{a+m} @>>> 0\\
@. @VVV @VVV \\
@. 0 @. 0\\
\endCD
$$
where  the two columns are central extensions (as $(a+m) + (b-m) \ge b$ and $a + (b-m)\ge b$).
Now, $N_1/Q_b$ and $N_2/Q_b$ are subgroups of $Q_a/Q_b$ whose images in $Q_a/Q_{b-m}$ and intersections with
$Q_{a+m}/Q_b$ coincide.  Suppose their intersections with $Q_{b-m}/Q_b$ are both equal
to $R/Q_b$ for some $R\supset Q_b$.  The discrepancy between the groups $N_i/Q_b$  is encoded by a homomorphism $N_1/Q_{b-m} \to Q_{b-m}/R$ which is trivial
 on $(N_1\cap Q_{a+m})/Q_{b-m}$, i.e., a homomorphism from a subgroup of $Q_a/Q_{a+m}$ to a quotient group
  of $Q_{b-m}/Q_b$.  The number of such homomorphisms is  bounded by $\|\bbf|^{c_2m^2}$ and the same applies
to the number of possibilities  $R$.  Thus parts (ii) of the proposition are also proved.
\hfill $\square$

We can now finish the proof of the upper bounds in Theorem 2.1.  The only thing left is to extend Propositions 2.5
 and 2.6, to the case that $G$ is not necessarily simple.  So let $G$ be as in Theorem 2.1
  parts (ii) and (iii).  By passing to an open subgroup of $M$, which is permissible by (1.2), we can assume that $M$
   is a product $\mathop{\Pi}\limits^r_{i=1} M_i$ where $M_i$ is an open compact subgroup of $S_i(k)$ and where  $S_i$
is a simple $k$-algebraic group.

We can now apply Proposition 1.5, with Propositions 2.5  and  2.6  to finish the
proof.
\hfill $\square$
\vskip.2cm
\noindent
{\it Proof of Theorem D}.
Theorem D is a special case of Theorem 2.1; the only point to note is that since $\calg_\eta$ is simply connected, $\Delta = \calg(\calo)$
is finitely generated [BL].
\hfill $\square$


\section{The global case}

In this section we prove first the following Theorem 3.1, from which Theorem C is deduced.  We then
 prove a variant of it, Proposition 3.2 below, to be used in the proof of Theorem B.
 \begin{thm} Let $k$ be a global field of characteristic $p\ge 0, \; S$ a non-empty set of valuations of
  $k$ containing all the archimedean ones and
  \begin{equation*}
  \calo_S=\{ x \in k| v(x) \ge 0, \forall v \notin s\}.
  \end{equation*}
  Let $\calg$ be a smooth group scheme over $\calo_S$, whose generic fiber $\calg_\eta$ is connected, simple
   and simply connected.  Let $H$ be the profinite group $\calg (\hat\calo_S)$.
 Then the growth type of $t_n(H)$ is:
\begin{enumerate}
\item[\rm (i)] $n$ if $\calg $ is of type $G_2, F_4$ or $E_8$.

\item[\rm (ii)] $n^{\log n/(\log\log n)^2}$ if $\calz(\calg_\eta)\neq 1$ and $p \nmid |\calz(\calg_\eta)|$

\item[\rm (iii)] $n^{\log n}$ if $p \Big| |\calz(\calg_\eta)|$.
   \end{enumerate}
   \end{thm}
   \begin{proof}
 We remark first that $H = \calg(\hat\calo_S)$ is indeed a  finitely generated group.
 This can be proved
 by analyzing the Frattini subgroup of $\calg(\hat\calo_S) = \mathop\Pi\limits_{v\notin S}\calg(\calo_v)$,
   or, alternatively,  as follows:  for sufficiently large $S' \supset S$, the $S'$-arithmetic group $\calg(\calo_{S'})$
    is finitely generated [Be] and dense in $\calg(\hat\calo_{S'})$ [PR], and $\calg(\hat\calo_S) =
    \calg(\hat\calo_{S'})\times\mathop{\Pi}\limits_{v\in S'\setminus S} \calg(\calo_v)$.
    As each of $\calg(\calo_v)$ is finitely generated [BL], so is $\calg(\hat\calo_S)$.

  Case (iii) follows easily from Theorem 2.1, i.e., already the projection to one
   local factor gives growth type at least $n^{\log n}$.  This, in turn, is the maximal
   possible normal subgroup growth type of finitely generated profinite groups by (1.6).

   We first prove the lower
   bounds of (i) and (ii).  For (i): $\calo_S$ (and hence $\hat \calo_S$) has at least $c n$ ideals of index at most $n$,
    for some fixed $c > 0$ (depending on $\calo_S$) and $n$ sufficiently large.  Each such ideal
$I$ gives rise to a principal congruence  subgroup
\begin{equation*} Q(I) =Ker \big(\calg(\hat\calo_S)\to\calg(\hat\calo_S/I)\big)
\end{equation*}
of index at most $n^d$ for some constant $d$.  This shows that the growth type of $t_n(H)$ is at least
 $n$.

 (ii)  Let $k'$ be a finite Galois extension of $k$, in which $\calz = \calz(\calg_\eta)$
 splits.  In particular $\calz(k')$ is a finite group of order say $z$,
 and by our assumption in (ii), $p\nmid z$.
Let $\calp_1$ be the set of primes in $k$ which splits completely
in $k'$ and $\calp = \calp_1\setminus S$.  By the Cebotarev density theorem, $\calp_1$
(and as $S$ is finite, also $\calp$) has positive density.

For a large real number $x$, let $\calp_x$ be the set of all primes in $\calp$ of norm
at most $x$ (where a norm $|P|$ of a prime $P$
 is its index in $\calo_S$).  By the Prime Number Theorem  and the
 positive density of $\calp$,  we have: $
\pi(x)/\frac{x}{\log x}$ and $ \psi(x)/x$ are both  bounded away from zero and infinity when
$\pi(x) = | \calp_x|$ and $\psi(x) = \suml_{P\in\calp_x} \log |P|$.

Let $m(x) = \mathop{\Pi}\limits_{P\in\calp_x}  P$ and let $Q\big(m(x)\big)$ denote, as before, the principal congruence
subgroup $\big(\!\!\!\mod m(x)\big)$.  It follows  that $|\calo_S/m(x)|\approx c^x_1$ for some constant $c_1$.
Fix now a prime $q$ dividing $z$.
Now:
\begin{equation*} H/Q\big(m(x)\big)\simeq \calg(\calo_S/m(x)\big) = \prod\limits_{P \in \calp_x}
\calg(\calo_S/P).
\end{equation*}
This shows that $Q\big(m(x)\big)$ is of index at most $c^x_2$ for some constant $c_2$.
On the other hand, for each $P \in \calp_x$, the finite group $\calg(\calo_S/P)$ has a central subgroup of order $z$, and hence
 also a central cyclic subgroup of order $q$.  Hence $H/Q\big(m(x)\big)$ has a central subgroup which is a $q$-elementary
  abelian group of rank $\pi(x)$.
  This shows that $H/Q\big(m(x)\big)$ has at least $q^{\frac{1}{4} \pi(x)^2}$ central
  subgroups and hence $H$ has at least  $q^{\frac{1}{4}\pi(x)^2}
  \ge
  q^{\frac{1}{4c^2}\cdot\frac{x^2}{(\log x)^2}}$
  normal subgroups
of index at most $c^x_2$.
This proves that the normal subgroup
  growth rate of $H$ is at least $n^{\log n/(\log\log n)^2}$.

  We turn now to the proof of the upper bound.  We start with both cases, (i) and (ii), together.
  We assume without loss of generality that $G$ is connected (see (1.2)).

  We have to prove an upper bound for  $t_n\big(\calg(\hat\calo_S)\big)$.
  Note that $\calg(\hat\calo_S)=
\mathop{\Pi}\limits_{v\notin S}\calg (\calo_v)$.
Let $\calp$ denote the set of all primes of $k$ which are not in $S$.
Let $\calp_1$ be the set of all $v \in \calp$ such that:
\begin{itemize}
\item[(a)] $\calg(\bbf_v)$ is an almost simple group, where
$\bbf_v = \calo_v/m_v$ and  $m_v$ is the maximal ideal of $\calo_v$.

\item[(b)]   If $Q_v(r) = Ker (\calg(\calo_v) \to \calg \big(\calo_v/m^r_v)\big)$ then
$[Q_v(1), Q_v(i)] = Q_v(i+1)$ for every $i \ge 1$.

\item[(c)] The elementary abelian $p$-group $Q_v(1)/Q_v(2)$ is a simple $\calg(F_v)$-module, and

\item[(d)] If $p = 0, \; v \nmid |\calz(\calg)|$
\end{itemize}
Now, unless  $\calg_\eta$ is of Ree type $\calp_1$ contains almost all primes in $\calp$.
By [SGA 3, XIX 2.5] all but finitely many fibers of $\calg$ are simple.  This implies (a).  For (b)
we use (1.8) and the logarithm map discussed in Section 2.  For (c), we note that every composition
factor of the adjoint representation of the special fiber is $|\bbf_v|$-restricted except if $|
\bbf_v|=2$ and $\calg_\eta$ is a form of $SL_2$.  By Steinberg's theorem [St], any restricted
 irreducible representation is irreducible over $\calg(\bbf_v)$.  As $\calg_\eta$ is not of Ree type and $p\nmid
  z$, the irreducibility  of the adjoint representation follows from (1.8).  Part (d) is clear.

Leaving  aside for now the two exceptional cases, consider  $S_1 = S\cup\{ v|v\notin\calp_1\}$ and $H_1=\calg
 (\hat\calo_{S_1}) = \mathop{\Pi}\limits_{v\in\calp_1} \calg(\calo_v)$.
One proves by induction  that for every open normal subgroup $N$ of $H_1$,
there exists an  ideal $I$
 in $\calo_{S_1}$ such that $Q_1 (I) \subseteq
 N\subseteq Z_1(I)$ when
 \begin{equation*}
Q_1 (I) = Ker \big(\calg(\hat\calo_{S_1})\to \calg(\hat \calo_{S_1}/I)\big)
\end{equation*}
and $Z_1(I)$ is the preimage in $H_1$ of the center of $\calg(\hat\calo_{S_1}/I)$.
It now follows, by a similar computation to the one carried out  above for the lower bound, that the
normal subgroup growth rate of $H_1$ is $n$ in case (i) and $n^{\log n/(\log\log n)^2}$ in case (ii).

Now $H= H_1\times H_2$ where $H_2=\mathop{\Pi}\limits_{v\notin\calp_1\cup S}\calg(\calo_v)$.
This is a product of finitely many groups.  The normal subgroup growth rate of $H_2$
is  at most polynomial by Theorem 2.1, Proposition 1.5 and Proposition 2.5.

 We can now finish the proof with the help of Proposition 1.5:
 In case (i), note that the only open
 normal subgroups of $H_1$ are the principal congruence subgroups $Q_1 (I)$, and $H_1/Q_1(I)$
 has no center.  So $t_n\big(\calg(\hat\calo_S)\big) \le t_n(H_1)^2 \cdot t_n (H_2)^2$,
 and so it is polynomially bounded.
 In case (ii), the normal subgroups of $H_1$ lie between
  $Q_1(I)$ and $Z_1(I)$.  Note that if $I$ is an ideal of index $n^\vare$, which is   a product
   of $m$ prime powers then $m\le c\frac{\log n}{\log\log n}$ (by the prime number theorem).
   So the center
   $Z_1(I)/Q_1(I)$ is of order at most $z^{c \log n/\log\log n}$ and its number
    of generators is at most $c'\log n /\log\log n$ (in fact $c'\le 2 c$).
    This shows that $z_n(H_1)\le
z^{c\log n/\log\log n}$ and $\delta_n(H_1) \le c'\log n/\log\log n$ where $z_n$ and $\delta_n$
    are as in (1.6).  Thus $t_n(H) \le t_n(H_1)^2 t_n(H_2)^2
    \cdot z^{cc'(\log n/\log\log n)^2}$.
 As $z $ is a constant, we have finished the proof except
     for groups of Ree type.

For these two cases, let us make the following remarks.
As before we decompose $H = H_1\times H_2$ where $H_1 = \mathop{\Pi}\limits_{v\in\calp_1} \calg(\calo_v)$
 and $\calp_1$ is the set of all primes $v$ for which $\calg(\calo_v)$ is hyperspecial.
 By [SGA3, XIX 2.5], for almost all primes $\calg_{\bbf_v}$ is simple and this implies that $\calg(\calo_v)$ is hyperspecial.
 As before $H_2$ is a finite product of local groups, and in this case the factors are of Ree type.
 By Theorem 2.1(ii), $t_n(H_2)$ is polynomially bounded.  By Proposition 1.5, $t_n(H) \le t_n(H_1)^2
  t_n(H_2)^2z_n(H_1)^{\delta_n(H_1)}$.
  As $H_1$ is a product of hyperspecial factors, its quotient by any open normal subgroup has trivial center by Proposition
   2.6(iii).  It suffices, therefore, to prove that $t_n(H_1)$ is polynomially bounded.

   Let $N\in T_n(H_1)$.
   Then there exists an ideal $I$ such that $N\supset Q_1 (I)$.
We claim that $I$ can be chosen to be so that $Q_1(I)$ has index at most $n^{c_1}$ for some constant $c_1$.
Indeed choose first $I = \mathop{\Pi}\limits^r_{i=1} P_i^{e_i}$ to be some ideal so that $N
\supset Q_1(I)$.  So $N/Q_1(I)$ is a normal subgroup of $S=\mathop{\Pi}\limits^r_{i=1}
\calg(\calo_{P_i}/P^{e_i}_i)$.  Now, as the centers of all the quotients of $S$ are trivial, we
 deduce from (1.4) that $N$ is a product of its intersections with the factors.  In the proof of Proposition 2.6
 we have analyzed the normal subgroups of hyperspecial groups of Ree type, and we implicitly showed
  that every normal subgroup  of index $r$ contains a principal congruence subgroup whose index
   in the first principal congruence subgroup is at most $r^2$.  This shows that $N$ contains a principal
    congruence subgroup of index at most $n^3$.
For each $I = \mathop{\Pi}\limits^r_{i=1} P^{e_i}_i$, the numbers
     of normal subgroups of $H_1$ containing $Q_1(I)$
     is   the product over $i$ of the number of  normal subgroups
of $H_i/Q_1(P_i^{e_i})$ which is at the number of  most $|H_1/Q_1(P_i^{e_i})|^\gamma\le |P_i^{e_i}|^{\gamma'}$ for a constant
 $\gamma'$.  Thus  $t_n(H_1)\le\suml_{\{I\big| |I|\le n^c\}} |I|^{\gamma'}$ and this is polynomially
  bounded.  Theorem 3.1 follows.
\end{proof}

For  use in \S 5,  let us put on record the following Proposition, whose proof is quite
similar to the proof of the lower bound of case (ii) of Theorem 3.1.

\begin{prop}  Let $k' \subset k$ an extension of global field, $S$ a finite set of primes
 in $k$ (containing all the archimedean ones) and $S'$ the corresponding induced primes of $k'$.  Let
 $\calo_S$ (resp. $\calo_{S'}$) be the ring of $S$-integers in $k$ (resp. $S'$-integers in $k'$).
Let $\calg$ be a smooth group scheme defined over $\calo_S$ and $\calg'$ a connected smooth group scheme defined over $\calo_{S'}$,
 such that  $\calg'\underset{\calo_{S'}}{\times} \calo_S = \calg^\circ$.
 Assume
the generic fiber $\calg'_\eta$ is simply connected,  $\calz(\calg'_\eta)\neq
 \{ 1\}$ and $ p \nmid |\calz(\calg'_\eta)$.
 Assume $H$ is a subgroup of $\calg(\hat\calo_S)$ containing $\calg'(\hat\calo_{S'})$ as a normal open subgroup.
Then $t_n(H)\succcurlyeq  n^{\log n/(\log\log n)^2}$.
\end{prop}
\begin{proof} By Theorem 3.1, $t_n(\calg'(\hat\calo_{S'}))\succcurlyeq n^{\log n/(\log\log n)^2}$.
Recall that we have shown there that for suitable choices of product of primes of $\calo_{S'}$,
 $m(x) = \mathop{\Pi}\limits_{P\in\calp_x} P$, there is a sufficiently large $q$-elementary abelian
  central subgroup $V = \mathop{\Pi}\limits_{P \in \calp_x} C_P$  in $\calg'\big(\calo_{S'}/m(x)\big)$ where $C_P$ is the $q$-part of the center
of $\calg'(\calo_{S'}/P)$.  These provide enough normal subgroups to ensure that growth.
   The principal congruence subgroups of $\calg'(\hat\calo_{S'})$ are intersections of principal
    congruence subgroups of $\calg(\hat\calo_{S})$ with $\calg'(\hat\calo_{S'})$ and are therefore
     normalized by $H$.  $H$ also normalizes the individual factors $\calg'(\calo_v)$ for $
v \in S'$.  Hence it preserves $C_P$ for every $P \in \calp_x$.
All the factors $C_P$ are isomorphic and $\calg'(\hat\calo_{S'})$ acts trivially.  There are finitely
 many possible homomorphisms from $H/\calg'(\hat\calo_{S'})$ to $Aut(C_P)$.  Hence the action is diagonal on a
 sufficiently large subset of $\calp_x$.  This gives the desired lower bound for $H$ as well.
 \end{proof}

We finally note:
\vskip.2cm
\noindent
{\it Proof of Theorem C.}  Theorem C is an immediate corollary of Theorem 3.1.
Indeed, $\Delta$ is infinite so the classical strong approximation theorem [PR] implies that
$\Delta =\calg(\calo_S)$ is dense
 in the profinite group $\calg(\hat\calo_S)$ and the profinite topology of $\calg(\hat\calo_S)$
  induces on $\Delta$ the congruence topology, so $D_n(\Delta) = t_n\big(\calg(\hat\calo_S)\big)$.
\hfill $\square$

\section{Specializing while preserving the Zariski closure}

This section is devoted to the following question:   Let $A$ be an integral domain with
fraction field $K$ and  $\Gamma$ a finitely generated subgroup of $GL_n(A)$ with Zariski closure $G$
 in $GL_{n, K}$. Is there a specialization $\phi:A\to k$, where $k$ is a global field, such that
 the Zariski closure of $\phi(\Gamma)$ is $\ol K$-isomorphic to $G$?

 Of course, we cannot expect this to be true for every $\Gamma$.  For example, if $G$
  (as an algebraic group over an algebraic closure $\ol K$ of $K$) is not isomorphic to a group defined over
   some global field, then the Zariski closure of $\vp(\Gamma)$ cannot be isomorphic to $G$.  Recall,
    for example, that there are uncountably many $\bbc$-isomorphism classes of unipotent algebraic groups, so most of them
     are clearly not isomorphic to groups defined over global fields.  For our purposes,
     it suffices to consider the case where $G$ is connected and (absolutely) simple.
     In this case, as  is well known,
$G$ is $\ol K$-isomorphic to a group defined over the prime field, so potentially our question may
       have a positive answer.  This is exactly what we prove in the following theorem.  In fact, a similar
result holds for semisimple groups and even for reductive groups, but the proof for the simple case is
considerably
easier  and sufficient for our needs.
\begin{thm} Let $A$ be an integral domain, finitely generated over the prime field of characteristic
 $p \ge 0$, with fraction field $K$.  Let $\Gamma \le GL_n(A)$ denote a finitely generated subgroup
whose Zariski closure in $GL_{n, K}$ is a connected absolutely simple group $G$.
Then there exists a global field $k$ and a ring homomorphism $\phi\colon A\to k$
 such that the Zariski closure of $\phi(\Gamma)$ in $GL_{n, k}$ is $\ol K$-isomorphic to $G$.
 \end{thm}
 \begin{remark*}  Note that we assert that the groups are isomorphic, but we do not claim that the ambient
 representations of the Zariski closures are isomorphic (i.e., we do  not claim
 that they are conjugate in $GL_n(\ol K)$).
 This is because in characteristic $p> 0$, representations of a simple algebraic group need not be rigid.
\end{remark*}

We begin with a few general remarks about Zariski closure.
 If $Y\to S$ is a morphism of schemes, $X\subset Y(S)$ is a set of sections and $s$
is a point of $S$, the closure of $X\cap Y_s$ in the fiber $Y_s$ is contained in $\ol X \cap Y_s$ since the latter is closed in $Y_s$.
If $S$ is irreducible and $s \in S$ is the generic point, then closure commutes with restriction to
 $Y_s$.  Indeed, any closed set in $Y$ containing $x_s$ for $x$ an element of $Y(S)$ contains
  all of $x$ and therefore any closed set in $Y$ containing $X\cap Y_s$ contains $\ol X$.

  A second remark is that if $Y\to S$ is a morphism of schemes,  $X $ is a subset of $Y(S)$ and $T\to S$
   is an open morphism, then the Zariski closure of $X\underset{S}{\times} T$ in $Y\underset{S}{\times}
   T$ equals $(\ol X\underset{S}{\times} T)^{red}$,
    i.e., set-theoretically, the Zariski closure commutes with open base change ([EGA IV, 2.4.11]).

    In particular this is the case when $T\to S$ is \'etale as well as the case when $T$ is obtained from $S$
     by tensoring by  an arbitrary field extension [EGA IV, 2.4.10].
\vskip.2cm

     We break the proof into several lemmas, in which we keep the notations of Theorem 4.1.  Let
$     \calg$ be the Zariski closure of $\Gamma$ in $GL_{n, A}$.  Note that the generic fiber of $\calg$ is $G$.
\begin{lem} There exists an \'etale $A$-algebra $B$ such that the Zariski closure $\calg_B$
of $\Gamma$ in $GL_{n, B}$ is a {\it split} simple group scheme.
 \end{lem}
 \begin{proof}  By construction, $\calg$ is affine and finitely presented.  By [EGA IV, 9.7.7], after
 inverting some element of $A$, we may assume the fibers are geometrically integral, and by
 generic flatness, we may also assume that $\calg$ is flat.
 By [SGA3, XIX 2.5], by inverting an additional element, we may further assume that $\calg$
  is a simple group scheme; and by [EGA IV, 6.12.6, 6.13.5], we may assume $A$ is integrally closed.
  Thus, by [SGA3, XXII 2.3], there exists an \'etale $A$-algebra $B$ such that $\calg \underset{A}{\times} B$
   is a split simple group scheme.  So by the remark preceding the lemma,
   $\calg_B=(\calg\underset{A}{\times} B)^{red}=
\calg\underset{A}{\times} B$.
\end{proof}

Note that an \'etale extension of a normal  integral domain is a direct sum of integral domains [SGA1, I 9.2].
Replacing $A$ by any
summand of $B$, we obtain an algebra satisfying the hypotheses of Theorem 4.1 and in addition we
can assume from now on that $\calg$ is split.

\begin{lem}  Given a simple algebraic group $H$ over an algebraically closed field, there exists a finite
 set of proper subgroups $H_1, \dots, H_r$ such that every positive-dimensional proper subgroup of $H$
  is conjugate to a subgroup of $H_i$ for some $i = 1, \dots, r$.
\end{lem}
\begin{proof} See Liebeck-Seitz [LS1, LS2, LS3].
\end{proof}
 \begin{lem}  There exists an open subset of $Spec(A)$ such that for every point $s$ of the subset,
  the closure of $\Gamma $ in $\calg(\Bbbk_s)$ is either finite or all of $\calg_s$ (By $\Bbbk_s$ we mean
  the residue field of $Spec(A)$ at $s$).
  \end{lem}
  \begin{proof}  Let $Y$ denote the disjoint union of $r$ copies of $G\underset{K}{\times} \ol K$, where $r$ is
the number of conjugacy classes of maximal positive dimensional subgroups as in Lemma 4.3.  $Y$ can be thought of as
parametrizing the maximal positive dimensional subgroups of $G$.  Let $Z \subset G\underset {K}{\times} Y$ denote
the $Y$-subgroup scheme such that $Z_y \subset G$ is the subgroup parametrized  by $y \in Y$.
Let $\tau_1,\dots,\tau_\ell$ denote a finite set of generators for $\Gamma$, which we regard as
morphisms from $Spec(A)$ to $G$ and hence also as sections of $G\underset{K}{\times} Spec(A)\to
Spec(A)$.

Let $p_i, \; i = 1, 2, 3$ denote the projection map from $G\underset{K}{\times}Spec(A) \underset{K}{\times} Y$
 to the product in which the $i$-th factor is omitted.
The intersection
\begin{equation*}
W=\mathop{\cap}\limits^\ell_{j=1} p_1\left(p_2^{-1} (Z) \cap p^{-1}_3\left(\tau_j(Spec A)\right)\right)
\end{equation*}
as a subset of $Spec(A)\underset{K}{\times} Y$ is the set of ``bad points", i.e., the set of pairs $(\phi, y),
 \; \phi \in Spec(A)$ and $y \in Y$ such that  $\phi(\Gamma)\subseteq Z_y$.
 By Chevalley's  theorem, $W$ is a constructible set.  The same is true for its projection $\ol W$
 to $Spec(A)$.  Now, $\ol W$ omits the generic point (as $\Gamma$ is dense in $G_K$) and hence it omits
  a non-empty affine open set $U$ in $Spec(A)$.
  This proves the lemma.
\end{proof}

>From now on, we will replace $A$ by the coordinate ring of $U$.
Let $\bbf$ be the field of constants in $A$, i.e., the algebraic closure of the prime field in $A$.
Let us fix an (absolutely) irreducible  (almost faithful) representation $\rho$ of $G\hookrightarrow GL_m$
 defined over $A$.
 (Such a representation exists  over the prime field and can then be extended   to $A$).
 We define a character $\chi \colon \Gamma \to A$ by $\chi (\gamma) = tr\rho(\gamma), \gamma \in \Gamma$.
By Burnside's  Lemma  [CR, 36.1], $\chi(\Gamma)$ is infinite.  Now, if $\chi(\Gamma)\subset \bbf$, since
 specialization is injective on constants, for any specialization, $\phi(\Gamma)$ contains elements with
 infinitely many different character values, so $\phi(\Gamma)$ cannot lie within a finite group.   Therefore
  by Lemma 4.4, it is dense in $G$.  Otherwise, fix $\gamma$ such that $\chi(\gamma)\in A$ is non-constant.
   If $p=0$, as $A$ is finitely generated over $\bbq,  \; \chi(\gamma)-r$ is not invertible in $A$ for sufficiently large $r \in \bbn$.
    We choose large $r > m$ and a specialization $\phi$ such that $\phi(\chi(\gamma)) = r$.  A sum of $m$
    roots of unity cannot equal $r$, so $\phi(\gamma)$ is of infinite order.  Thus $\phi(\Gamma)$ is not
     finite and again we finish by Lemma 4.4.  Let now $p > 0$: In this
      case we regard $\chi$ as a dominant morphism from $Spec(A)$ to $\bba^1$. Let $C$
be a quasi-section [EGA IV, 17.16.1], i.e., a curve in $Spec(A)$ such that $\chi|_C$ is still dominant.  Let $k$ be the function field of $C$
        and $\phi$ the specialization from $A$ to $k$.  By construction, $\phi(\chi(\gamma))$ is not constant, and we are done by Lemma 4.4.
Theorem 4.1 is therefore proved.
\hfill $\square $

\section{Proof of Theorem B}

We are now ready to reap  the fruits of our labor and  prove Theorem B.
We will use the notation of  the introduction.

So let $\Gamma \le GL_n(F)$ be a finitely generated group Zariski dense in $G$, \  $S(\Gamma) = \ol G/Z(\ol G)$  and $S(\Gamma) = \mathop{\Pi}\limits^r_{i=1} S_i$.
If $r=0, \; G$ and $\Gamma$ are virtually  solvable and we are done.
Assume thus that $r > 0$ and some $S_i$, say $S_1$,  of type $X$, satisfies  $\calz(\tilde S_1) \neq \{ 1\}$.
Replacing $G$ by a suitable quotient if needed (which may also entail changing $n$), we can assume that $G^\circ$ itself is a product $\mathop{\Pi}\limits^r_{i=1} S_i$,
 all
 the factors $S_i$  are adjoint and isomorphic to one another, all of type $X$, and that $G/G^\circ$ acts transitively on the set
 of factors.
 So now $\Gamma$ and $G$ are subgroups of $G^\circ\rtimes Out (G^\circ)$.  Pulling
 back to $\widetilde{G^\circ} \rtimes Out (G^\circ)$, replacing $\Gamma$ by $\tilde \Gamma$, we can
 assume that each $S_i$ is simply connected.
 Let $\Gamma^\circ = \Gamma\cap G^\circ$ and let $\Gamma^1$ be the projection of $\Gamma^\circ$ to $S_1$.
 Note that $S_1$ can   also be regarded as a subgroup of $G^\circ$ and therefore of $GL_n$.
 As $\Gamma^\circ$ is of finite index in $\Gamma$, \ $\Gamma^1$ is also a finitely generated group.  There is therefore an
  integral domain $A$ in $F$ which is finitely generated over the prime field such that both $\Gamma$ and $\Gamma_1$
  are inside $GL_n(A)$.  The Zariski closure of $\Gamma_1$  is the connected absolutely simple group $S_1$.
  By Theorem 4.1, there exists a global field $k$ and a ring homomorphism $\phi\colon
   A\to k$ such that the Zariski closure of $\phi(\Gamma_1)$ in $GL_{n, k}$ is $F$-isomorphic
    to $S_1$.  The specialization $\phi$  induces also a homomorphism from $\Gamma $ to $GL_n(k)$.  Let
  $H$ be the Zariski closure of $\phi(\Gamma)$ in $GL_{n,k}$, and $H^\circ$ the connected component of $H$.
  It follows that $H^\circ = \mathop{\Pi}\limits^s_{j=1} R_j$ where $1\le s \le r$ and for every $j$, \
  $1\le j \le s, \; R_j$ is isomorphic to $S_1$.  Moreover, $H/H^\circ$ acts transitively on the set
  $\{ R_j\}^s_{j=1}$.

  Replacing $\Gamma$ by a suitable quotient   we may assume it is contained in $GL_n(k)$
  and its Zariski closure $G$ satisfies $G^\circ = \mathop{\Pi}\limits^r_{i=1}S_i$, each $S_i$
  is a simply connected  group
   of type $X$ and $G/G^\circ$ acts transitively on $\{ S_i\}^r_{i=1}$.

Now, as $\Gamma$ is finitely generated,  there exists a finite set of primes $S$ in $k$ (containing all the archimedean
ones) such that $\Gamma$ is in $GL_{n}( \calo_S)$.
Let $\calg$ be the Zariski closure of $\Gamma$ in $GL_{n, \calo_S}$.  Thus $\Gamma \subset\calg(\calo_S)$
 and $\calg_\eta =G$.
 Assume
  now that we are in case (ii)  of Theorem B, i.e., $p | \,\big | \calz( S_i)|$
  for some $i$.  Let $v$ be a prime outside $S$, so $\Gamma$ is in $\calg (\calo_v)$ and
   is Zariski dense in $\calg_{k_v}$.  Let $M$ be the closure of $\Gamma$ in the profinite
    group $\calg (\calo_v)$.  We can apply Theorem 2.1 to deduce that  the normal subgroup
     growth of $M$ is at least $n^{\log n}$.  It follows that the same applies to $\Gamma$.  Now, by (1.6),
      this is the maximal possible normal subgroup growth type; hence Theorem B(ii) is proved.

      To prove (i), we continue as follows:
      By [P2] we can find a global subfield $k'$ of $k$ and
 a semisimple, connected, simply connected algebraic group $G'$ over $k'$
       such that $G'\underset{k'}{\times} k$ is isogenous to $G^\circ$ and $[\Gamma^\circ,\Gamma^\circ]$
       is contained as a Zariski dense subgroup
of $G' (k')$ satisfying strong approximation.
As $p\nmid |\calz(G^\circ)|$, the isogeny is an
isomorphism.  Let $\calg'$ be a smooth group scheme defined over the $S'$-integers of $k'$
for some finite set $S'$ of primes of $k'$, such that $\calg'_\eta = G'$.  As $\calg'\underset{\calo_{S'}}{\times} \calo_S$ and $\calg^\circ$
 have isomorphic generic fibers, after enlarging $S$ and $S'$ we may assume these group schemes are isomorphic and $S'$
  is the restriction of $S$ to $k'$. Let $H$ be the closure of $\Gamma$ in $\calg(\hat\calo_S)$.
  By strong approximation, and by further enlarging $S$ and $S'$ if needed,
  we can assume that $[\Gamma^\circ,\Gamma^\circ]$ is dense in $\calg'(\hat\calo_{S'})
  \subseteq \calg(\hat\calo_S)$.
We claim $[H\colon\calg'(\hat\calo_{S'})]$ is finite, or equivalently, $[\Gamma^\circ\colon \Gamma^\circ\cap\calg'(k')]$ is finite.
Indeed, let $G'=\calg'_\eta$, ${G'}^{ad}=G'/\calz(G')$ and $\pi\colon G'\to
{G'}^{ad}$ the universal
 cover map.  Again by [P2], $\pi(\Gamma^\circ) \subset {G'}^{ad}(k')$.  As $\Gamma_0$ is finitely generated
  and ${G'}^{ad}(k')/\pi\big(G'(k')\big)$ is a torsion abelian group, the image of the former in the latter
   is finite.
  By Proposition 3.2, $t_n(H) \succcurlyeq n^{\log n/(\log\log n)^2}$ and hence the
  same is true for $\Gamma$.
 \hfill $\square $
\vskip.2cm
\noindent
{\it Acknowledgement.}  The authors are indebted to D. Hold, E. Hrushovski, A. Rapinchuk, D. Segal and V.N. Venkataramana for some helpful
 discussions.  Some of this work was done during visits to Yale University, University of Chicago
 and The Hebrew University.  Their hospitality and support are gratefully acknowledged, as well as the support
  of the NSF and the US-Israel Binational Science Foundation.


\author{
Michael Larsen \hfil $\qquad \qquad \, $ Alexander Lubotzky
\\
Department of Mathematics\hfil $\quad$  Institute of Mathematics
\\
Indiana University \hfil $\qquad \qquad \,  $ The Hebrew University
\\
Bloomington, Indiana 47405 USA \hfil \! \! \! \! \!\!\!\!\!\!\!\!\!\!\!\!\!\! Jerusalem 91904, Israel
\\
larsen@math.indiana.edu \hfil $\qquad$  \!\!\!\!\!alexlub@math.huji.ac.il}
\end{document}